\documentclass{amsart}
\usepackage{titlesec}
\numberwithin{equation}{section}
\titleformat{\section}[block]{\Large\bfseries\filcenter}{\thesection.}{1em}{}
\titleformat{\subsection}[runin]{\Large\mdseries\filright}{\indent\thesubsection}{0.2cm}{}
\renewcommand{\thesubsection}{(\alph{subsection})}
\newtheoremstyle{theorem}{}{}{\itshape}{}{\scshape}{}{5pt}{}
\theoremstyle{theorem}
\newtheorem*{mainthm}{Main Theorem}
\newtheorem{thm}{Theorem}[section]

\newtheorem{lem}[thm]{Lemma}
\title{starshaped locally convex hypersurfaces with prescribed curvature and boundary}
\author{Chenyang Su}
\begin{document}
\large
\maketitle
\vspace{-6mm}
\begin{abstract}
In this paper we find strictly locally convex hypersurfaces in $\mathbb{R}^{n+1}$ with prescribed curvature and boundary. The main result is that if the given data admits a strictly locally convex radial graph as a subsolution, we can find a radial graph realizing the prescribed curvature and boundary. As an application we show any smooth domain on the boundary of a compact strictly convex body can be deformed to a smooth hypersurface with the same boundary (inside the convex body) and realizing any prescribed curvature function smaller than the curvature of the body. \end{abstract}

\section{Introduction}
In this paper we study a classical problem in differential geometry: given a collection $\Gamma=(\Gamma_1, \ldots, \Gamma_k) $ of embedded
codimension 2 submanifolds of  $\mathbb{R}^{n+1}$, find a hypersurface $\Sigma$ with prescribed curvature and boundary data. That is, we seek to solve 
\begin{equation} 
f(\kappa_1(\Sigma), \ldots, \kappa_n(\Sigma))=\psi(X),\,\, \partial \Sigma=\Gamma~,
\end{equation}
where $X$ is the position vector of $\Sigma$.

 The solvability of the problem in this generality still remains quite open. However if we confine ourselves to strictly locally convex hypersurfaces, i.e., oriented hypersurfaces with positive principal curvatures everywhere, the theory of fully nonlinear elliptic PDE  becomes a powerful tool to study the solvability  of this geometric problem.
However, even in the strictly locally convex case, we still have to deal with the geometric nature of this problem carefully. The geometry of the prescribed boundary plays a crucial role in the solvability. For a K-surface (surface of constant Gauss curvature $K$ in $\mathbb{R}^3$), an obvious necessary condition for the solvability is that the boundary curve cannot have inflection point. In higher dimension, this corresponds to the condition that the second fundamental form of the boundary is nondegenerate everywhere. These conditions are very far from being sufficient as there are subtle obstructions to solvability; for more discussion, see \cite{hoffmanrosenbergspruck}, \cite{rosenberg} and \cite{ghomiphdthesis}. One method to avoid the obstructions of the boundary comes from a PDE point of view. In \cite{radialgs} and \cite{Guannonconvex}, Guan and Spruck (extending the work of Caffarelli, Nirenberg and Spruck \cite{CNS1}) proved the existence of a smooth solution for  general Monge-Amp\`ere equations on smooth domains with multiple boundary components and  arbitrary geometry, if there exists a  subsolution.  The authors also proved the same theorem for radial graphs over a domain in $\mathbb{S}^n$. With the help of Perron Method, in \cite{gsgaussperron}, \cite{trudingerwang}, \cite{locallyconvexgs}, \cite{smith1},  \cite{clarkesmith} and \cite{smith2} the authors further extended these results to its full parametric generality for strictly locally convex hypersurfaces of constant curvature.\\

In this paper we continue the study of the locally convex case but allow variable prescribed curvature. Let $\Omega$ be a smooth domain in $\mathbb{S}^n$  with boundary $\partial\Omega$ (which may have multiple components), $\psi$ a positive smooth function on the cone $\Lambda=\left\{X\in\mathbb{R}^{n+1}|\frac{X}{\|X\|}\in\bar\Omega\right\}$ and $\phi$ a positive smooth function on $\partial\Omega$. We seek a smooth, strictly locally convex hypersurface $\Sigma$ that can be represented as a radial graph
  \begin{equation}
  X(x)=\rho(x)x,\quad\rho>0,\:x\in\bar\Omega,
  \end{equation}
  with prescribed curvature
  \begin{equation}\label{0.2}
  f(\kappa_\Sigma[X])=\psi(X),
  \end{equation}
  and boundary values
  \begin{equation}\label{0.3}
  X(x)=\phi(x)x,\quad x\in\partial\Omega,
  \end{equation}
  where $\kappa_\Sigma[X]=(\kappa_1,...,\kappa_n)$ denotes the principal curvatures of $\Sigma$ at $X$ with respect to the inward unit normal. The  curvature function  $f(\lambda)$ is a positive smooth symmetric function defined on the convex cone $\Gamma_n^+\equiv\{\lambda\in\mathbb{R}^n: \text{each component}\: \lambda_i>0\}$ satisfying the fundamental structure conditions
  \begin{equation}\label{0.4}
  f_i(\lambda)\equiv\frac{\partial f(\lambda)}{\partial{\lambda}_i}>0\quad\text{in}\:\Gamma_n^+,\: 1\leq i\leq n,
  \end{equation}
  and
  \begin{equation}\label{0.5}
  f~ \text{is a concave function}.
  \end{equation}
  In addition, $f$ is assumed to satisfy 
  \begin{equation}\label{0.6}
  f=0\text{ on }\partial\Gamma_n^+,
  \end{equation}
  \begin{equation}\label{0.7}
  \sum f_i(\lambda)\lambda_i\geq\sigma_0 \quad\text{ on } \{\lambda\in\Gamma_n^+: \psi_0\leq f(\lambda)\leq\psi_1\},
  \end{equation}
  for any $\psi_1>\psi_0>0$, where $\sigma_0$ is a positive constant depending on $\psi_0$ and $\psi_1$. In addition we will need the following more technical assumption: for every $C>0$ and every compact set $E$ in $\Gamma_n^+$, there exists $R=R(E, C)>0$ such that
  \begin{equation}\label{0.8}
  f(\lambda_1,...,\lambda_{n-1},\lambda_n+R)\geq C\quad\text{ for all }\lambda\in E.
  \end{equation}
  Curvature functions satisfying \eqref{0.4}-\eqref{0.8} include a large family of examples, which can be found in \cite{locallyconvexgs}. However we shall point out the curvature quotient $S_{n,k}^{\frac{1}{n-k}}=(S_n/S_k)^{\frac{1}{n-k}}$ ($S_k$ is the k-th elementary symmetric polynomial) does not satisfy \eqref{0.8}.
  We will also assume that
  \begin{equation}\label{0.9}
  \Omega\text{ does not contain any hemisphere}
  \end{equation}
  and there exists a smooth, strictly locally convex radial graph $\Sigma'$: $\bar X(x)=\bar\rho(x)x$ over $\bar\Omega$ satisfying
  \begin{equation}\label{0.10}
  \begin{aligned}
  f(\kappa_{\Sigma'}[\bar X])&>\psi(\bar X) \quad \text{on} \,\Omega,\\
              \bar\rho&=\phi\quad\text{on}\:\partial\Omega.
  \end{aligned}
  \end{equation}
  Our main theorem is stated as follows:
  \begin{mainthm}\label{0.11}
  Under conditions \eqref{0.4}-\eqref{0.10} there exists a smooth, strictly locally convex hypersurface $\Sigma$ which is a radial graph $X(x)=\rho(x)x$, with $\rho\leq\bar\rho$, satisfying equations \eqref{0.2}-\eqref{0.3}. Moreover, $c<\rho<\bar\rho$ in $\Omega$ and the principal curvatures $\kappa_i$ of $\Sigma$ satisfies $C^{-1}<\kappa_i<C$, where $c$, $C$ are both uniform positive constants.
  \end{mainthm}
  
  We can apply the Main Theorem to some more geometric settings.
  \begin{thm}\label{0.12}
  Let $\psi$ be a smooth positive function on $\mathbb{R}^{n+1}$, $f$ be the same as in Main Theorem. $\Gamma=\{\Gamma_1,\Gamma_2,...,\Gamma_m\}$ is a disjoint collection of smooth, closed, embedded, codimension $2$ submanifolds in $\mathbb{R}^{n+1}$. Suppose $\Gamma$ bounds a smooth convex hypersurface $\Sigma'$ satisfying $f(\kappa_{\Sigma'}[X])>\psi(X)$ for all $X\in\Sigma'$, and $\Sigma'$ lies strictly on one side of every tangent hyperplane of $\Sigma'$ at $\Gamma$. Then $\Gamma$ bounds a smooth convex hypersurface $\Sigma$ satisfying $f(\kappa_{\Sigma}[X])=\psi(X)$ for all $X\in\Sigma$.  
  \end{thm}
  \begin{proof}
  Since $\Sigma'$ is convex and satisfies $f(\kappa_{\Sigma'}[X])>\psi(X)>0$, from condition \eqref{0.6} we see $\Sigma'$ is also strictly locally convex. From Theorem 1.2.5 in \cite{ghomiphdthesis}, we see that $\Sigma'$ can be extended to an ovaloid $O$. Choose a point $A\in O\setminus\Sigma'$ as the new origin. Then $\Sigma'$ is a radial graph over a smooth domain $\Omega$ in $\mathbb{S}^n$. Moreover, $\Omega$ is a proper subset of a hemisphere of $\mathbb{S}^n$. To see this, we just need to note that $\Sigma'$ stays strictly on one side of the tangent hyperplane of $O$ at $A$. Therefore, we can apply  Main Theorem to construct $\Sigma$.
  The hypersurface $\Sigma$ is convex as well as strictly locally convex. We can glue $O\setminus\Sigma'$ and $\Sigma$ together to get a closed hypersurface $O'$, though nonsmooth. $O'$ is locally convex in the sense that for every point in $O'$ there exists an open neighborhood staying on the boundary of a convex body. In \cite{heijenoort} Van Heijenoort proved in the generalized sense a complete locally convex hypersurface in $\mathbb{R}^{n+1}$ which is absolutely convex at least at one point is the boundary of an n+1-dimensional convex set. Here $O'$ clearly satisfies all the assumptions. Therefore $\Sigma$, which is part of $O'$, is a convex hypersurface.  
  \end{proof}

  Similarly, we can give a different version of Theorem \ref{0.12} which generalizes the results in \cite{radialgs} and \cite{clarkesmith}. 
  
  \begin{thm}\label{0.13}
  In $\mathbb{R}^{n+1}$ let $B$ be an n+1-dimensional compact convex body with a smooth boundary $\partial B$, $\psi$ and $f$ the same as in Theorem \ref{0.12}. Suppose $f(\kappa_{\partial B}[X])>\psi(X)$ for all $X\in\partial B$. Then for every smooth subdomain $D$ of $\partial B$ with nontrivial boundary $\Gamma$, $\Gamma$ bounds a strictly locally convex hypersurface $\Sigma$ inside $B$ which satisfies $f(\kappa_\Sigma[X])=\psi(X)$ for all $X\in\Sigma$. Moreover, the set bounded inside $\Sigma\cup(\partial B\setminus D)$ is also a convex body.
  \end{thm}
  
  The proof of Theorem $\ref{0.13}$ is the same as Theorem $\ref{0.12}$.\\
  
  The proof of the Main Theorem is organized as follows. In Section \ref{1.18} we introduce some preliminary results and define two elliptic operators to express \eqref{0.2}. One of these is to be used in Section \ref{2.93} to derive a priori $C^2$ estimates for the solution. We use this in Section \ref{3.1} and work with the other operator, applying the continuity method and degree theory to prove the existence of solution. \\
  
  The author would like to thank his advisor Professor Joel Spruck for introducing this problem and giving helpful suggestions.  
   
  \hspace{3cm}
  \section{Preliminaries and defining the equations}\label{1.18}
    Consider the radial graph $\Sigma$: $X(x)=\rho(x)x$, $x\in\Omega\subset\mathbb{S}^n$. We set $u=\frac{1}{\rho}$. Let $e_1,...,e_n$ be a smooth orthonormal frame field on $\Omega$. $\{e_i\}$ will be pushed forward to a frame field $\{\tau_i\}$ on $\Sigma$, where $\tau_i=-\frac{\nabla_iu}{u^2}x+\frac{1}{u}e_i$. Here $\nabla$ is the Riemannian connection on $\mathbb{S}^n$ and $\tilde\nabla$ is the connection in $\mathbb{R}^{n+1}$. We adopt the notation $\nabla_iu$, $\nabla_{ij}u=\nabla_i\nabla_ju$, $\nabla_{ijk}u=\nabla_i\nabla_j\nabla_ku$ etc for the covariant derivatives of a function. The metric of the radial graph $\Sigma$ can be given in terms of $u$ by
  \begin{equation}\label{1.1}
  g_{ij}=\langle\tau_i,\tau_j\rangle=\frac{1}{u^2}\delta_{ij}+\frac{1}{u^4}\nabla_iu\nabla_ju,
  \end{equation}
  where $\langle\cdot,\cdot\rangle$ denotes the standard inner product in $\mathbb{R}^{n+1}$. The inward unit normal to $\Sigma$ is
  \begin{equation}\label{1.2}
  \nu=\frac{-\nabla u-ux}{w},
  \end{equation}
  where $\nabla u=\text{grad}u$, $w=\sqrt{u^2+|\nabla u|^2}$. The second fundamental form of $\Sigma$ is
  \begin{equation}\label{1.31}
  h_{ij}=\frac{1}{uw}(u\delta_{ij}+\nabla_{ij}u).
  \end{equation}
  Therefore $\Sigma$ is strictly locally convex if and only if the matrix
  \begin{equation}\label{1.12}
  [u\delta_{ij}+\nabla_{ij}u] \text{ is positive definite at any point.}
  \end{equation}
  Later for convenience we may just say a function is strictly locally convex without specifying what we really mean is that the corresponding radial graph is strictly locally convex, e.g., when we say $u$ is strictly locally convex, that means \eqref{1.12} holds. The principal curvatures of $\Sigma$ are the eigenvalues of $h_{ik}g^{kj}$, which is similar to the symmetric matrix $[a_{ij}]:=[g^{ij}]^\frac{1}{2}[h_{ij}][g^{ij}]^\frac{1}{2}$. $[g^{ij}]^\frac{1}{2}$ can be written as $u[\gamma^{ij}]$, where $[\gamma^{ij}]$ and its inverse matrix $[\gamma_{ij}]$ are given by
  \begin{equation}\label{1.4}
  \gamma^{ij}=\delta_{ij}-\frac{\nabla_{i}u\nabla_{j}u}{w(u+w)},\quad \gamma_{ij}=\delta_{ij}+\frac{\nabla_{i}u\nabla_{j}u}{u(u+w)}.
  \end{equation}
  Therefore
  \begin{equation}\label{1.3}
  a_{ij}=\frac{u}{w}\gamma^{ik}(u\delta_{kl}+\nabla_{kl}u)\gamma^{lj},
  \end{equation}
  and its eigenvalues are the principal curvatures of $\Sigma$.  

  Now we do a reformulation of equation \eqref{0.2} in the form
  \begin{equation}\label{1.5}
  G(\nabla^2u,\nabla u,u)=\psi(X(x)).
  \end{equation}
  Let $\mathcal{S}$ be the set of $n\times n$ symmetric matrices and $\mathcal{S}^+=\{A\in\mathcal{S}:A>0\}$, i.e., the set of positive definite symmetric matrices. With the function $F$ defined by
  \begin{equation}\label{1.6}
  F(A)=f(\lambda(A)),\quad A\in\mathcal{S}^+,
  \end{equation}
  where $\lambda(A)$ denotes the eigenvalues of $A$, equation \eqref{0.2} thus can be written in the form
  \begin{equation}\label{1.7}
  F([a_{ij}])=\psi(X(x)).
  \end{equation}
  Therefore, the function $G$ in \eqref{1.5} is defined by
  \begin{equation}\label{1.8}
  G(\nabla^2u,\nabla u,u)=F([a_{ij}]).
  \end{equation}
  Then equations \eqref{0.2}-\eqref{0.3} can be rewritten as
  \begin{equation}\label{1.14}
  \begin{aligned}
  G(\nabla^2u,\nabla u,u)&=\psi(X(x))\quad\text{in}\:\Omega,\\
  u&=\varphi\quad\text{on}\:\partial\Omega,
  \end{aligned}
  \end{equation}
  where $\varphi=\frac{1}{\phi}$.

  We next recall some properties of $F$ and $G$ (see \cite{CNS3}, \cite{locallyconvexgs}). We will use the notation
  \begin{equation}\label{1.9}
  F^{ij}(A)=\frac{\partial F}{\partial a_{ij}}(A).
  \end{equation}
  The matrix $[F^{ij}(A)]$ is symmetric and has eigenvalues $f_1,...,f_n$. By assumption \eqref{0.4}, $[F^{ij}(A)]$ is therefore positive definite for $A\in\mathcal{S}^+$, while \eqref{0.5} implies that $F$ is a concave function in $\mathcal{S}^+$. $[F^{ij}(A)]$ and $A$ can be orthogonally diagonalized simultaneously. Consequently, we have
  \begin{equation}\label{1.10}
  F^{ij}(A)a_{ij}=\sum f_i\lambda_i.
  \end{equation}
  For equation \eqref{1.14}, we have
  \begin{equation}\label{1.11}
  G^{ij}=\frac{\partial G}{\partial\nabla_{ij}u}=\frac{u}{w}F^{kl}\gamma^{ik}\gamma^{lj}.
  \end{equation}
  So equation \eqref{1.14} is elliptic if \eqref{1.12} holds. The concavity of $F$ implies that $G$ is concave with respect to $\nabla_{ij}u$. From the assumption \eqref{0.10}, the function $\underline{u}=\frac{1}{\bar{\rho}}$ is a subsolution of equation \eqref{1.14}, i.e.,
  \begin{equation}\label{1.13}
  \begin{aligned}
  G(\nabla^{2}\underline{u},\nabla \underline{u}, \underline{u})&=\underline{\psi}(x)>\psi(\bar X(x))\quad \text{in}\:\bar\Omega,\\
  \underline{u}&=\varphi\quad\text{on}\:\partial\Omega.
  \end{aligned}
  \end{equation}

  Here we chose $u=\frac{1}{\rho}$ to set up equation \eqref{1.14} because it turns out the operator $G$ works very well for deriving a priori estimates. However, when it comes to applying the continuity method and degree theory to prove the existence of the solution, we find \eqref{1.14} is not the right equation to work with. The trouble mainly comes from, as our computation will show later, the fact that $G_u$ is positive and can not be bounded easily. So here it is necessary for us to express \eqref{0.2}-\eqref{0.3} in a different form. We set $v=-\ln\rho=\ln u$. Then $[a_{ij}]$ can be written in terms of $v$, that is,
  \begin{equation}\label{1.19}
  a_{ij}=\frac{e^v}{w}(\delta_{ij}+\gamma^{ik}\nabla_{kl}v\gamma^{lj}),
  \end{equation}
  where
  \begin{equation}\label{1.20}
  w=\sqrt{1+|\nabla v|^2},\quad \gamma^{ij}=\delta_{ij}-\frac{\nabla_iv\nabla_jv}{w(1+w)}.
  \end{equation}
  Then equation \eqref{1.14} becomes
  \begin{equation}\label{1.21}
  \begin{aligned}
  H(\nabla^2v,\nabla v,v)=F([a_{ij}])&=\psi(X(x))\quad\text{in}\:\Omega,\\
  v&=\ln\varphi\quad\text{on}\:\partial\Omega.
  \end{aligned}
  \end{equation}
  Here we shall note equations \eqref{1.14} and \eqref{1.21} will appear in different sections. So the ambiguous notations $w$ and $\gamma^{ij}$ shall not cause any confusion. Correspondingly $\underline{v}=\ln\underline{u}$ is the subsolution with respect to \eqref{1.21}. We call $v$ strictly locally convex if
  \begin{equation}
  [\delta_{ij}+\nabla_iv\nabla_jv+\nabla_{ij}v]>0\quad\text{in}\:\bar\Omega.
  \end{equation}
  $H$ is elliptic for strictly locally convex functions $v$ and is concave with respect to $\nabla_{ij}v$.
  \hspace{3cm}
  \section{A priori estimates}\label{2.93}
  In this section we derive the a priori $C^2$ estimates for locally strictly convex solutions $u$ of equation \eqref{1.14} with $u\geq\underline{u}$. The $C^1$ bound follows from the convexity of the radial graph, which is established in \cite{radialgs}. We recall the results here.
  \begin{thm}\label{1.22}
  Let $u\geq\underline{u}$ be a strictly locally convex function with $u=\underline{u}$ on $\partial\Omega$. Then we have the estimates
  \begin{equation}\label{1.23}
  K^{-1}\leq u\leq K,\quad |\nabla u|\leq C\quad\text{in}\:\bar\Omega,
  \end{equation}
  where $K$ depends on $\Omega$, $\sup_{\partial\Omega}\underline{u}$, $\inf_\Omega\underline{u}$ and $C$ depends in addition on $\|\underline{u}\|_{C^2(\bar\Omega)}$.
  \end{thm}
   We remark here that getting the upper bound of $u$ is the only place we need condition \eqref{0.9}. From now on we fix $K$ as above. Later when we use $K$, it always means the same constant. We define $\Lambda_K:=\{X\in\Lambda|K^{-1}\leq\|X\|\leq K\}$.

  Therefore we devote the rest of this section to deriving bounds for $\nabla^2u$. In the rest of this section, $u$ will be a smooth strictly locally convex solution of \eqref{1.14} with $u\geq\underline{u}$. We shall remark here later for the proof of existence, we will work on auxiliary forms of equation \eqref{1.14}, i.e., we will change $\psi$ to some other functions. So in this section the reader shall think $\psi$ as a general positive smooth function on $\Lambda$, not just only the prescribed curvature function in \eqref{0.2}.
  
  \subsection{\textit{Bound for $|\nabla^2u|$ on $\partial\Omega$}.} Given a point $x_0\in\partial\Omega$, let $e_1,...,e_n$ be a local orthonormal frame field on $\mathbb{S}^n$ around $x_0$, obtained by parallel translation of a local orthonormal frame field on $\partial\Omega$ and the interior, unit, normal vector field to $\partial\Omega$, along the geodesic perpendicular to $\partial\Omega$ on $\mathbb{S}^n$. $e_n$ is the parallel translation of the unit normal field on $\partial\Omega$.

  On $\partial\Omega$ we have $u-\underline{u}=0$ so that
  \begin{equation}\label{2.5}
  \nabla_{\alpha}(u-\underline{u})=0,\quad\nabla_{\alpha}(\nabla_{\beta}(u-\underline{u}))=0\quad \text{for}\:\alpha,\beta<n,
  \end{equation}
  and hence
  \begin{equation}\label{2.6}
  \begin{aligned}
  \nabla_{\alpha\beta}(u-\underline{u})&=\nabla_{\alpha}(\nabla_{\beta}(u-\underline{u}))-\sum_i\langle\nabla_\alpha e_\beta,e_i\rangle\nabla_i(u-\underline{u})\\
  &=-B_{\alpha\beta}\nabla_n(u-\underline{u})\quad \text{for}\:\alpha,\beta<n,
  \end{aligned}
  \end{equation}
  where $B_{\alpha\beta}=\langle\nabla_\alpha e_\beta,e_n\rangle$ is the second fundamental form of $\partial\Omega$.
  It follows that
  \begin{equation}\label{2.7}
  |\nabla_{\alpha\beta}u(x_0)|\leq C,\quad \alpha,\beta<n,
  \end{equation}
  where $C$ depends on $\Omega$, $\inf_\Omega\underline{u}$ and $\|\underline{u}\|_{C^2(\bar\Omega)}$.

  We now proceed to estimate the mixed, normal, tangential derivatives $\nabla_{n\alpha}u(x_0),\:\alpha<n$. We need some properties of the linearized operator
  \begin{equation}
  L=G^{ij}\nabla_{ij}+G^s\nabla_s,
  \end{equation}
  where $G^s\equiv\frac{\partial G}{\partial\nabla_su}$. We also denote $\frac{\partial G}{\partial u}$ by $G_u$.

  \begin{lem}
  For some constant $C>0$ depending on $\Omega$, $\inf_\Omega\underline{u}$, $\|\underline{u}\|_{C^2(\bar\Omega)}$ and $\sup_{\Lambda_K}\psi$, we have
  \begin{equation}\label{2.8}
  \sum|G^s|\leq C,
  \end{equation}
  \begin{equation}\label{2.9}
  |G_u|\leq C(1+\sum G^{ii}).
  \end{equation}
  \end{lem}
  \begin{proof}
  \eqref{2.8} and \eqref{2.9} follow from straightforward computation. Since $G(\nabla^2u,\nabla u,u)=F([\frac{u}{w}\gamma^{ik}(u\delta_{kl}+\nabla_{kl}u)\gamma^{lj}])$,
  \begin{equation}\label{2.11}
  \begin{aligned}
  G^s&=F^{ij}u(u\delta_{kl}+\nabla_{kl}u)\frac{\partial}{\partial\nabla_su}\left(\frac{1}{w}\gamma^{ik}\gamma^{lj}\right)\\
  &=-\frac{\nabla_su}{w^2}F^{ij}a_{ij}+2F^{ij}\frac{u}{w}(u\delta_{kl}+\nabla_{kl}u)\gamma^{ik}\frac{\partial\gamma^{lj}}{\partial\nabla_su}.
  \end{aligned}
  \end{equation}
  Since $a_{ij}=\frac{u}{w}\gamma^{ik}(u\delta_{kl}+\nabla_{kl}u)\gamma^{lj}$,
  \begin{equation}\label{2.12}
  \frac{u}{w}\gamma^{ik}(u\delta_{kl}+\nabla_{kl}u)=a_{ik}\gamma_{kl}.
  \end{equation}
  We also compute
  \begin{equation}\label{2.13}
  \frac{\partial\gamma^{lj}}{\partial\nabla_su}=\frac{u\nabla_lu\nabla_ju\nabla_su}{w^3(u+w)^2}-\frac{\nabla_ju\gamma^{ls}}{w(u+w)}-\frac{\nabla_lu\gamma^{js}}{w(u+w)}.
  \end{equation}
  Then
  \begin{equation}\label{2.14}
  \gamma_{kl}\frac{\partial\gamma^{lj}}{\partial\nabla_su}=-\frac{\nabla_ju\gamma^{ks}}{w(u+w)}-\frac{\nabla_ku\gamma^{js}}{u(u+w)}
  \end{equation}
  since $\gamma_{kl}\nabla_lu=\frac{w}{u}\nabla_ku$, $\gamma_{kl}\gamma^{ls}=\delta_{ks}$.
  From \eqref{2.11}, \eqref{2.12}, \eqref{2.14}, and we also notice $[F^{ij}][a_{ij}]=[a_{ij}][F^{ij}]$ since $[F^{ij}]$ and $[a_{ij}]$ can be diagonalized simultaneously, we finally get
  \begin{equation}\label{2.15}
  G^s=-\frac{\nabla_su}{w^2}F^{ij}a_{ij}-2F^{ij}a_{ik}\frac{\nabla_ju\gamma^{ks}}{uw}.
  \end{equation}
  From the established $C^1$ bound \eqref{1.23} and \eqref{1.10}, conditions \eqref{0.4}-\eqref{0.6} and the fact that $[F^{ij}]$ and $[a_{ij}]$ can be diagonalized simultaneously, we get
  \begin{equation}\label{2.20}
  \sum|G^s|\leq C\sum f_i\kappa_i\leq Cf(\kappa_1,...,\kappa_n)=C\psi\leq C.
  \end{equation}
  \eqref{2.8} is established.

  For $G_u$, we do same kind of computation,
  \begin{equation}\label{2.16}
  \begin{aligned}
  G_u&=F^{ij}\Bigg(\frac{1}{w}\gamma^{ik}(u\delta_{kl}+\nabla_{kl}u)\gamma^{lj}-\frac{u}{w^2}\gamma^{ik}(u\delta_{kl}+\nabla_{kl}u)\gamma^{lj}\frac{\partial w}{\partial u}\\
  &\quad\quad\quad\:+\frac{u}{w}(u\delta_{kl}+\nabla_{kl}u)\frac{\partial(\gamma^{ik}\gamma^{lj})}{\partial u}+\frac{u}{w}\gamma^{ik}\gamma^{lj}\delta_{kl}\Bigg)\\
  &=F^{ij}\left(\frac{1}{u}a_{ij}-\frac{u}{w^2}a_{ij}+2a_{ik}\gamma_{kl}\frac{\partial\gamma^{lj}}{\partial u}+\frac{u}{w}\gamma^{ik}\gamma^{kj}\right),
  \end{aligned}
  \end{equation}
  and
  \begin{equation}\label{2.17}
  \frac{\partial\gamma^{lj}}{\partial u}=\frac{\nabla_lu\nabla_ju}{w^3},
  \end{equation}
  \begin{equation}\label{2.18}
  \gamma^{ik}\gamma^{kj}=\delta_{ij}-\frac{\nabla_iu\nabla_ju}{w^2}.
  \end{equation}
  Combining \eqref{2.16}-\eqref{2.18}, we get
  \begin{equation}\label{2.19}
  G_u=\left(\frac{1}{u}-\frac{u}{w^2}\right)F^{ij}a_{ij}+2F^{ij}a_{ik}\frac{\nabla_ju\nabla_ku}{uw^2}+\frac{u}{w}F^{ij}\left(\delta_{ij}-\frac{\nabla_iu\nabla_ju}{w^2}\right).
  \end{equation}
  Similarly as \eqref{2.8} and also by \eqref{1.11} and the fact that $[F^{ij}]$ and $[G^{ij}]$ are positive definite, we get \eqref{2.9}.
  \end{proof}

  \begin{lem}\label{2.34}
  For some positive constants $t$ and $\delta$ sufficiently small and $N$ sufficiently large depending on $\Omega$, $\inf_\Omega\underline{u}$, $\|\underline{u}\|_{C^2(\bar\Omega)}$, $\sup_{\Lambda_K}\psi$ and convexity of $\underline{u}$, the function $v=u-\underline{u}+td-Nd^2$ satisfies
  \begin{equation}\label{2.21}
  \begin{cases}
  Lv\leq-1-\beta\sum{G^{ii}}&\text{in }\Omega\cap B_\delta\\
  v\geq0&\text{on }\partial(\Omega\cap B_\delta),
  \end{cases}
  \end{equation}
  where $\beta>0$ depends only on the convexity of $\underline{u}$, $d$ is the distance function to $\partial\Omega$ and $B_\delta$ is a ball of radius $\delta$ centered at a point on $\partial\Omega$.
  \end{lem}
  \begin{proof}
  By the convexity of the surface $\bar{X}=\frac{1}{\underline{u}}x$, we can find $\beta>0$ such that
  \begin{equation}\label{2.22}
  [\underline{u}\delta_{ij}+\nabla_{ij}\underline{u}]\geq4\beta I\quad\text{on }\bar\Omega.
  \end{equation}
  Thus
   \begin{equation}\label{2.28}
   \lambda(\underline{u}\delta_{ij}+\nabla_{ij}\underline{u}-3\beta\delta_{ij}) \text{ lies in a compact set of } \Gamma_n^+. \end{equation}
   Since $|\nabla d|=1$ and $-C_1I\leq[\nabla_{ij}d]\leq C_1I$ where $C_1$ only depends on the geometry of $\partial\Omega$, we have
  \begin{equation}\label{2.23}
  Ld=G^{ij}\nabla_{ij}d+G^i\nabla_id\leq C\sum{G^{ii}}+\sum|G^i|\quad\text{in }\Omega\cap B_\delta
  \end{equation}
  and
  \begin{equation}\label{2.24}
  \lambda(\nabla_{ij}\underline{u}+\underline{u}\delta_{ij}+N\nabla_{ij}d^2-2\beta\delta_{ij})\geq\lambda(\nabla_{ij}\underline{u}+\underline{u}\delta_{ij}+2N\nabla_id\nabla_jd-3\beta\delta_{ij})\quad\text{in }\Omega\cap B_\delta
  \end{equation}
  when we make $2N\delta<\frac{\beta}{C_1}$.

  Next, from the concavity of $F$ and the fact that $u\geq\underline{u}$,
  \begin{equation}\label{2.25}
  \begin{aligned}
  &\quad\: F\left(\left[\frac{u}{w}\gamma^{ik}(\underline{u}\delta_{kl}+\nabla_{kl}\underline{u}+N\nabla_{kl}d^2-2\beta\delta_{kl})\gamma^{lj}\right]\right)-\psi(X(x))\\
  &=F\left(\left[\frac{u}{w}\gamma^{ik}(\underline{u}\delta_{kl}+\nabla_{kl}\underline{u}+N\nabla_{kl}d^2-2\beta\delta_{kl})\gamma^{lj}\right]\right)-F\left(\left[\frac{u}{w}\gamma^{ik}(u\delta_{kl}+\nabla_{kl}u)\gamma^{lj}\right]\right)\\
  &\leq\frac{u}{w}F^{ij}\gamma^{ik}\gamma^{lj}(\underline{u}\delta_{kl}+\nabla_{kl}\underline{u}+N\nabla_{kl}d^2-2\beta\delta_{kl}-u\delta_{kl}-\nabla_{kl}u)\\
  &=G^{kl}\nabla_{kl}(\underline{u}-u+Nd^2)+(\underline{u}-u)\sum{G^{ii}}-2\beta\sum{G^{ii}}\\
  &\leq G^{kl}\nabla_{kl}(\underline{u}-u+Nd^2)-2\beta\sum{G^{ii}}.
  \end{aligned}
  \end{equation}
  \begin{equation}\label{2.26}
  \begin{aligned}
  &\quad\:L(u-\underline{u}+td-Nd^2)\\
  &=G^{ij}\nabla_{ij}(u-\underline{u}-Nd^2)+t(G^{ij}\nabla_{ij}d+G^i\nabla_id)+G^i\nabla_i(u-\underline{u})-2NdG^i\nabla_id.
  \end{aligned}
  \end{equation}
  Now combine \eqref{2.23}-\eqref{2.26} and the established $C^1$ bound, we get
  \begin{equation}\label{2.27}
  \begin{aligned}
  &\quad\:L(u-\underline{u}+td-Nd^2)\\
  &\leq\psi(X(x))-F\left(\left[\frac{u}{w}\gamma^{ik}(\underline{u}\delta_{kl}+\nabla_{kl}\underline{u}-3\beta\delta_{kl}+2N\nabla_kd\nabla_ld)\gamma^{lj}\right]\right)\\
  &\quad+(Ct-2\beta)\sum{G^{ii}}+(C+2Nd+t)\sum|G^i|\\
  &=\psi(X(x))-f\left(\lambda\left(\frac{u}{w}\gamma^{ik}(\underline{u}\delta_{kl}+\nabla_{kl}\underline{u}-3\beta\delta_{kl})\gamma^{lj}+\frac{2Nu}{w}\gamma^{ik}\nabla_kd\nabla_ld\gamma^{lj}\right)\right)\\
  &\quad+(Ct-2\beta)\sum{G^{ii}}+(C+2Nd+t)\sum|G^i|.
  \end{aligned}
  \end{equation}
  By \eqref{2.28} and the established $C^1$ bound, there exists a uniform positive constant $\lambda_0$ such that
  \begin{equation}\label{2.29}
  \left[\frac{u}{w}\gamma^{ik}(\underline{u}\delta_{kl}+\nabla_{kl}\underline{u}-3\beta\delta_{kl})\gamma^{lj}\right]\geq\lambda_0I.
  \end{equation}
  We use an orthogonal matrix $P$ to diagonalize $[\frac{2u}{w}\gamma^{ik}\nabla_kd\nabla_ld\gamma^{lj}]$ to be $\mathrm{diag}\{0,0,...,\mu\}$, where by the $C^1$ bound $\mu\geq\mu_0$ for some uniform positive constant $\mu_0$. Therefore
  \begin{equation}\label{2.31}
  \begin{aligned}
  &P^T\left[\frac{u}{w}\gamma^{ik}(\underline{u}\delta_{kl}+\nabla_{kl}\underline{u}-3\beta\delta_{kl})\gamma^{lj}\right]P+
  N\mathrm{diag}\{0,0,...,\mu\}\\
  \geq&\mathrm{diag}\{\lambda_0,\lambda_0,...,\lambda_0+N\mu_0\}.
  \end{aligned}
  \end{equation}
  So by the Minimax Characterization Theorem and \eqref{0.4}, \eqref{0.8}, we get
  \begin{equation}\label{2.32}
  \begin{aligned}
  &\quad\: f\left(\lambda\left(\frac{u}{w}\gamma^{ik}(\underline{u}\delta_{kl}+\nabla_{kl}\underline{u}-3\beta\delta_{kl})\gamma^{lj}+\frac{2Nu}{w}\gamma^{ik}\nabla_kd\nabla_ld\gamma^{lj}\right)\right)\\
  &=f\left(\lambda\left(P^T\left[\frac{u}{w}\gamma^{ik}(\underline{u}\delta_{kl}+\nabla_{kl}\underline{u}-3\beta\delta_{kl})\gamma^{lj}+\frac{2Nu}{w}\gamma^{ik}\nabla_kd\nabla_ld\gamma^{lj}\right]P\right)\right)\\
  &\geq f(\lambda_0,\lambda_0,...,\lambda_0+N\mu_0)\to+\infty\quad\text{as }N\to+\infty.
  \end{aligned}
  \end{equation}
  Therefore by \eqref{2.8}, \eqref{2.27} and \eqref{2.32}, we can choose $t$ small enough such that $Ct\leq\beta$ and $N$ large enough such that
  \begin{equation}\label{2.33}
  \begin{aligned}
  &\quad\:L(u-\underline{u}+td-Nd^2)\\
  &\leq\psi(X(x))-f\left(\lambda\left(\frac{u}{w}\gamma^{ik}(\underline{u}\delta_{kl}+\nabla_{kl}\underline{u}-3\beta\delta_{kl})\gamma^{lj}+\frac{2Nu}{w}\gamma^{ik}\nabla_kd\nabla_ld\gamma^{lj}\right)\right)\\
  &\quad+(Ct-2\beta)\sum{G^{ii}}+(C+2Nd+t)\sum|G^i|\\
  &\leq-1-\beta\sum{G^{ii}}.
  \end{aligned}
  \end{equation}
  Finally we can make $\delta$ even smaller, that is, $\delta\leq\frac{t}{N}$. Then $u-\underline{u}+td-Nd^2\geq0$ on $\partial(\Omega\cap B_\delta)$.
  \end{proof}

  Now consider $Av+B\rho^2$, where $v$ is as in Lemma \ref{2.34}, $\rho$ as the distance function to $x_0$ and $A$, $B$ are large positive constants to be determined. We compute
  \begin{equation}\label{2.35}
  \begin{aligned}
  L\nabla_\alpha u&=G^{ij}\nabla_{ij\alpha}u+G^s\nabla_{s\alpha}u\\
  &=G^{ij}(\nabla_{\alpha ij}u+\delta_{ij}\nabla_\alpha u-\delta_{\alpha j}\nabla_iu)+G^s\nabla_{\alpha s}u\\
  &=\nabla_\alpha G(\nabla^2u,\nabla u,u)-G_u\nabla_\alpha u+\nabla_\alpha u\sum{G^{ii}}-G^{i\alpha}\nabla_i u\\
  &=\nabla_\alpha\psi(X(x))-G_u\nabla_\alpha u+\nabla_\alpha u\sum{G^{ii}}-G^{i\alpha}\nabla_i u\\
  &=\langle-\frac{\nabla_\alpha u}{u^2}x+\frac{1}{u}e_\alpha,\tilde\nabla\psi\rangle-G_u\nabla_\alpha u+\nabla_\alpha u\sum{G^{ii}}-G^{i\alpha}\nabla_i u,
  \end{aligned}
  \end{equation}
  where the standard formula for commuting the order of covariant derivatives on $\mathbb{S}^n$ is applied. Then by the established $C^1$ bound, \eqref{2.9} and the fact that $\underline{u}\in C^\infty(\bar\Omega)$,
  \begin{equation}\label{2.36}
  |L\nabla_\alpha(u-\underline{u})|\leq C\left(1+\sum{G^{ii}}\right).
  \end{equation}
  Therefore we can first pick $B$ large enough to ensure $Av+B\rho^2\geq\pm\nabla_\alpha(u-\underline{u})$ on $\partial(\Omega\cap B_\delta(x_0))$. Then by \eqref{2.21} and \eqref{2.36}, we can pick $A\gg B$ to ensure $L(Av+B\rho^2\pm\nabla_\alpha(u-\underline{u}))\leq0$ in $\Omega\cap B_\delta(x_0)$. By maximum principle, $Av+B\rho^2\geq\pm\nabla_\alpha(u-\underline{u})$ in $\Omega\cap B_\delta(x_0)$. We also notice $(Av+B\rho^2)(x_0)=\nabla_\alpha(u-\underline{u})(x_0)=0$. Thus $\nabla_n(-Av-B\rho^2)(x_0)\leq\nabla_{n\alpha}(u-\underline{u})(x_0)\leq\nabla_n(Av+B\rho^2)(x_0)$, which implies $|\nabla_{n\alpha}u(x_0)|\leq C$, where $C$ depends on $\Omega$, $\inf_\Omega\underline{u}$, $\|\underline{u}\|_{C^3(\bar\Omega)}$, $\|\psi\|_{C^1(\Lambda_K)}$ and the convexity of $\underline{u}$. The mixed, normal, tangential derivatives bound on the boundary is established.

  Now we move on to the pure normal derivative bound. First we prove
  \begin{equation}\label{2.37}
  M\equiv\min_{x\in\partial\Omega}\,\min_{\xi\in{T_x(\partial\Omega)},\,|\xi|=1}(u+\nabla_{\xi\xi}u)\geq c_0
  \end{equation}
  for some uniform $c_0>0$, where $T_x(\partial\Omega)$ denotes the tangent space of $\partial\Omega$ at $x\in\partial\Omega$.

  Following the idea of \cite{locallyconvexgs}, let $\sigma$ be a smooth defining function of $\Omega$, that is, $\sigma$ is defined in a neighborhood of $\Omega$ satisfying
  \begin{equation}\label{2.38}
  \Omega=\{\sigma<0\},\quad\partial\Omega=\{\sigma=0\},\quad\text{and }|\nabla\sigma|=1\text{ on }\partial\Omega.
  \end{equation}
  Note that $\nabla\sigma=-\mathbf{n}$ on $\partial\Omega$ where $\mathbf{n}$ is the interior unit normal to $\partial\Omega$ and
  \begin{equation}\label{2.39}
  \nabla_{\xi\xi}u=\nabla_{\xi\xi}\underline{u}-\mathbf{n}(u-\underline{u})\nabla_{\xi\xi}\sigma\quad\text{on }\partial\Omega
  \end{equation}
  for any $\xi$ tangent to $\partial\Omega$.

  Suppose $M$ is achieved at $x_0\in\partial\Omega$ with $\xi\in T_{x_0}(\partial\Omega)$. Same as in the beginning of this subsection, we construct a local orthonormal frame field $e_1,...,e_n$ around $x_0$ and make $e_1(x_0)=\xi$. Then by \eqref{2.39}
  \begin{equation}\label{2.40}
  M=u(x_0)+\nabla_{11}u(x_0)=\underline{u}(x_0)+\nabla_{11}\underline{u}(x_0)-\mathbf{n}(u-\underline{u})(x_0)\nabla_{11}\sigma(x_0).
  \end{equation}
  We may assume
  \begin{equation}\label{2.41}
  \mathbf{n}(u-\underline{u})(x_0)\nabla_{11}\sigma(x_0)>\frac{1}{2}(\underline{u}(x_0)+\nabla_{11}\underline{u}(x_0)),
  \end{equation}
  for otherwise we are done because of the strictly local convexity of the graph $\bar{X}$.

  Let $\zeta=(\zeta_1,...,\zeta_n)$ be defined as
  \begin{equation}\label{2.42}
  \begin{aligned}
  \zeta_1&=-\nabla_n\sigma\left((\nabla_1\sigma)^2+(\nabla_n\sigma)^2\right)^{-1/2},\\
  \zeta_j&=0,\quad2\leq j\leq n-1,\\
  \zeta_n&=\nabla_1\sigma\left((\nabla_1\sigma)^2+(\nabla_n\sigma)^2\right)^{-1/2},
  \end{aligned}
  \end{equation}
  in $\bar\Omega\cap B_\delta(x_0)$. Notice the well-definedness of $\zeta$ is ensured by \eqref{2.38} and a sufficiently small $\delta$. From \eqref{2.41} and since $\nabla_{ij}\sigma\zeta_i\zeta_j$ is continuous and $0\leq\mathbf{n}(u-\underline{u})\leq C$ on $\partial\Omega$, there exists $c_1>0$ and $\delta>0$ (which may be even smaller) such that
  \begin{equation}\label{2.43}
  \nabla_{ij}\sigma\zeta_i\zeta_j(x)\geq\frac{1}{2}\nabla_{ij}\sigma\zeta_i\zeta_j(x_0)=\frac{\nabla_{11}\sigma(x_0)}{2}>\frac{\underline{u}(x_0)+\nabla_{11}\underline{u}(x_0)}{4\mathbf{n}(u-\underline{u})(x_0)}\geq c_1\quad\text{in }\Omega\cap B_\delta(x_0).
  \end{equation}
  Thus the function $\Phi:=\frac{\underline{u}+\nabla_{ij}\underline{u}\zeta_i\zeta_j-M}{\nabla_{ij}\sigma\zeta_i\zeta_j}$ is smooth and bounded in $\Omega\cap B_\delta(x_0)$. Note on the boundary $\zeta=(1,0,...,0)$ and by \eqref{2.39},
  \begin{equation}\label{2.44}
  \underline{u}+\nabla_{ij}\underline{u}\zeta_i\zeta_j+(\nabla(u-\underline{u})\cdot\nabla\sigma)\nabla_{ij}\sigma\zeta_i\zeta_j=\underline{u}+\nabla_{11}u=u+\nabla_{11}u\geq M\quad\text{on }\partial\Omega\cap B_\delta(x_0).
  \end{equation}
  Therefore
  \begin{equation}\label{2.45}
  \Phi+\nabla(u-\underline{u})\cdot\nabla\sigma\geq0\quad\text{on }\partial\Omega\cap B_\delta(x_0).
  \end{equation}
  Next we apply the linearized operator $L=G^{ij}\nabla_{ij}+G^s\nabla_s$ again,
  \begin{equation}\label{2.46}
  \begin{aligned}
  &\quad\: L(\Phi+\nabla(u-\underline{u})\cdot\nabla\sigma)\\
  &=G^{ij}\nabla_{ij}\Phi+G^s\nabla_s\Phi+G^{ij}\nabla_{ij}(\nabla u\cdot\nabla\sigma)-G^{ij}\nabla_{ij}(\nabla\underline{u}\cdot\nabla\sigma)\\
  &\quad+G^s\nabla_s(\nabla u\cdot\nabla\sigma)-G^s\nabla_s(\nabla\underline{u}\cdot\nabla\sigma),
  \end{aligned}
  \end{equation}
  where the terms $|G^{ij}\nabla_{ij}\Phi|$, $|G^s\nabla_s\Phi|$, $|G^{ij}\nabla_{ij}(\nabla\underline{u}\cdot\nabla\sigma)|$ and $|G^s\nabla_s(\nabla\underline{u}\cdot\nabla\sigma)|$ are clearly controlled by $C(1+\sum{G^{ii}})$. We only need to compute
  \begin{equation}\label{2.47}
  \begin{aligned}
  &\quad\:G^{ij}\nabla_{ij}(\nabla u\cdot\nabla\sigma)+G^s\nabla_s(\nabla u\cdot\nabla\sigma)\\
  &=G^{ij}\nabla_{ij}(\nabla_ku\nabla_k\sigma)+G^s\nabla_s(\nabla_ku\nabla_k\sigma)\\
  &=G^{ij}(\nabla_{ijk}u\nabla_k\sigma+\nabla_ku\nabla_{ijk}\sigma+2\nabla_{ik}u\nabla_{jk}\sigma)+G^s\nabla_{sk}u\nabla_k\sigma+G^s\nabla_ku\nabla_{sk}\sigma\\
  &=\nabla_k\sigma(G^{ij}\nabla_{ijk}u+G^s\nabla_{sk}u)+2G^{ij}\nabla_{ik}u\nabla_{jk}\sigma+\nabla_ku(G^{ij}\nabla_{ijk}\sigma+G^s\nabla_{sk}\sigma),
  \end{aligned}
  \end{equation}
  where $|\nabla_ku(G^{ij}\nabla_{ijk}\sigma+G^s\nabla_{sk}\sigma)|$ is controlled by $C(1+\sum{G^{ii}})$. Same as \eqref{2.35},
  \begin{equation}\label{2.48}
  G^{ij}\nabla_{ijk}u+G^s\nabla_{sk}u=\langle-\frac{\nabla_ku}{u^2}x+\frac{1}{u}e_k,\tilde\nabla\psi\rangle-G_u\nabla_ku+\nabla_ku\sum{G^{ii}}-G^{ik}\nabla_iu,
  \end{equation}
  so $|\nabla_k\sigma(G^{ij}\nabla_{ijk}u+G^s\nabla_{sk}u)|$ is controlled by $C(1+\sum{G^{ii}})$. The last term to be controlled is $2G^{ij}\nabla_{ik}u\nabla_{jk}\sigma$. But we notice by \eqref{1.3}
  \begin{equation}\label{2.49}
  \nabla_{kl}u=\frac{w}{u}\gamma_{ki}a_{ij}\gamma_{jl}-u\delta_{kl},
  \end{equation}
  so
  \begin{equation}\label{2.50}
  G^{ij}\nabla_{ik}u=G^{ij}\left(\frac{w}{u}\gamma_{kl}a_{ls}\gamma_{si}-u\delta_{ki}\right)=F^{si}\gamma^{ij}\gamma_{kl}a_{ls}-uG^{kj}.
  \end{equation}
  Therefore $|2G^{ij}\nabla_{ik}u\nabla_{jk}\sigma|$ is also controlled by $C(1+\sum{G^{ii}})$, and thus $|G^{ij}\nabla_{ij}(\nabla u\cdot\nabla\sigma)+G^s\nabla_s(\nabla u\cdot\nabla\sigma)|$ is controlled by $C(1+\sum{G^{ii}})$. Above all,
  \begin{equation}\label{2.51}
  L(\Phi+\nabla(u-\underline{u})\cdot\nabla\sigma)\leq C\left(1+\sum{G^{ii}}\right)\quad\text{in }\Omega\cap B_\delta(x_0).
  \end{equation}

  Now we can apply Lemma \ref{2.34} to construct the barrier $Av+B\rho^2$ as before. Choose $A\gg B\gg1$ so that
  \begin{equation}\label{2.52}
  \begin{aligned}
  L(Av+B\rho^2+\Phi+\nabla(u-\underline{u})\cdot\nabla\sigma)&\leq0\quad\text{in }\Omega\cap B_\delta(x_0),\\
  Av+B\rho^2+\Phi+\nabla(u-\underline{u})\cdot\nabla\sigma&\geq0\quad\text{on }\partial(\Omega\cap B_\delta(x_0)).
  \end{aligned}
  \end{equation}
  Then by the maximum principle and the fact that $Av+B\rho^2+\Phi+\nabla(u-\underline{u})\cdot\nabla\sigma=0$ at $x_0$, we get
  \begin{equation}\label{2.53}
  A\nabla_nv(x_0)+\nabla_n\Phi(x_0)-\nabla_{nn}(u-\underline{u})(x_0)+\nabla_n(u-\underline{u})(x_0)\nabla_{nn}\sigma(x_0)\geq0,
  \end{equation}
  which implies $\nabla_{nn}u(x_0)\leq C$. We thus have established $|\nabla^2u|\leq C$ at $x_0$. Then the principal curvatures of $\Sigma$ at $X(x_0)$, which are the eigenvalues of $[\frac{u}{w}\gamma^{ik}(u\delta_{kl}+\nabla_{kl}u)\gamma^{lj}]$, also have an upper bound. By the compactness argument and condition \eqref{0.6}, we get that the principal curvatures at $X(x_0)$ also have a uniform positive lower bound, which in turn gives a uniform positive lower bound of the eigenvalues of $[u\delta_{kl}+\nabla_{kl}u]$ at $x_0$. Therefore \eqref{2.37} is established. So now for every $x\in\partial\Omega$, the eigenvalues of $[u\delta_{\alpha\beta}+\nabla_{\alpha\beta}u]_{\alpha,\beta\leq n-1}$ have an uniform positive lower bound, which finally implies an upper bound for $u+\nabla_{nn}u$. $|\nabla_{nn}u|\leq C$ on $\partial\Omega$ is established and hence the bound for $|\nabla^2u|$ on $\partial\Omega$, which depends on $\Omega$, $\inf_\Omega\underline{u}$, $\|\underline{u}\|_{C^4(\bar\Omega)}$, $\|\psi\|_{C^1(\Lambda_K)}$, $\inf_{\Lambda_K}\psi$ and the convexity of $\underline{u}$.

  \subsection{\textit{Global bound for $|\nabla^2u|$}.} In this subsection we derive the global $C^2$ bound. It suffices to estimate $\max\kappa_i$, the maximum of the principal curvatures of $\Sigma$.

Choose a local orthonormal frame $\{\tau_1, \tau_2,..., \tau_n\}$ on $\Sigma$. $\nu$ is the inward unit normal. $\tilde{\nabla}$ is the connection of the Euclidean Space $\mathbb{R}^{n+1}$. $\bar\nabla$ is the induced Riemannian connection on $\Sigma$. $h$ is the second fundamental form of $\Sigma$. $h_{ij}=h(\tau_i,\tau_j)=\langle\tilde{\nabla}_{\tau_i}\tau_j,\nu\rangle=-\langle\tau_j,\tilde{\nabla}_{\tau_i}\nu\rangle$. We adopt the notation $h_{ijk}=\bar\nabla_kh_{ij}$, $h_{ijkl}=\bar\nabla_{kl}h_{ij}=\bar\nabla_l\bar\nabla_kh_{ij}$, etc. For a function $v$ defined on $\Sigma$, we write $v_i=\bar\nabla_iv$, $v_{ij}=\bar\nabla_{ij}v$.

First we need the standard formulas for commuting the order of covariant derivatives of second fundamental form. Since $\Sigma$ stays in $\mathbb{R}^{n+1}$, we have the following formulas (see \cite{schoensimonyau}), 
\begin{equation}\label{2.100}
h_{ijk}=h_{ikj},
\end{equation}
\begin{equation}\label{2.101}
h_{iijj}-h_{jjii}=h_{jj}\sum_mh_{im}^2-h_{ii}\sum_mh_{jm}^2.
\end{equation}

Next we need to differentiate two important quantities on $\Sigma$. $\rho$ is the standard Euclidean distance to the origin, and set
\begin{equation}\label{2.102}
\beta(X)=-\frac{1}{\rho}\langle\nu,X\rangle,\quad X\in\Sigma.
\end{equation}
By \eqref{1.2}, we get
  \begin{equation}\label{2.54}
  \beta=\frac{\nabla u+ux}{w}\cdot x=\frac{u}{w},
  \end{equation}
  which has both a positive upper bound and positive lower bound by previous estimates. We have
\begin{equation}\label{2.103}
\begin{aligned}
2(\rho_i\rho_j+\rho\rho_{ij})&=\bar\nabla_{ij}\rho^2=\bar\nabla_{ij}\langle X,X\rangle=\tau_j\tau_i\langle X,X\rangle-\bar\nabla_{\tau_j}\tau_i\langle X,X\rangle\\
&=2\tau_j\langle\tau_i,X\rangle-2\langle\bar\nabla_{\tau_j}\tau_i,X\rangle=2\langle\tilde\nabla_{\tau_j}\tau_i,X\rangle+2\langle\tau_i,\tau_j\rangle-2\langle\bar\nabla_{\tau_j}\tau_i,X\rangle\\
&=2\langle\tilde\nabla_{\tau_j}\tau_i-\bar\nabla_{\tau_j}\tau_i,X\rangle+2\delta_{ij}=2\langle h_{ij}\nu,X\rangle+2\delta_{ij}\\
&=-2\rho\beta h_{ij}+2\delta_{ij}.
\end{aligned}
\end{equation}
Therefore
\begin{equation}\label{2.104}
\rho_{ij}=\frac{1}{\rho}\delta_{ij}-\beta h_{ij}-\frac{1}{\rho}\rho_i\rho_j.
\end{equation}
By \eqref{2.102} we have $\rho\beta=-\langle\nu,X\rangle$. Therefore
\begin{equation}\label{2.105}
\begin{aligned}
\rho_i\beta+\rho\beta_i&=\bar\nabla_i(\rho\beta)=-\tau_i\langle\nu,X\rangle=-\langle\tilde\nabla_{\tau_i}\nu,X\rangle-\langle\nu,\tau_i\rangle\\
&=\langle \sum_jh_{ij}\tau_j,X\rangle=\sum_jh_{ij}\langle\tau_j,\rho\tilde\nabla\rho\rangle\\
&=\sum_jh_{ij}\langle\tau_j,\rho(\bar\nabla\rho+\langle\tilde\nabla\rho,\nu\rangle\nu)\rangle=\sum_jh_{ij}\langle\tau_j,\rho\bar\nabla\rho\rangle\\
&=\rho\sum_jh_{ij}\rho_j.
\end{aligned}
\end{equation}
Therefore
\begin{equation}\label{2.106}
\beta_i=\sum_jh_{ij}\rho_j-\frac{\beta}{\rho}\rho_i.
\end{equation}
Differentiate \eqref{2.106} again, we get
\begin{equation}\label{2.107}
\begin{aligned}
\beta_{ij}&=\bar\nabla_j\left(\sum_kh_{ik}\rho_k-\frac{\beta}{\rho}\rho_i\right)\\
&=\sum_kh_{ikj}\rho_k+h_{ik}\rho_{kj}-\frac{1}{\rho}\beta_j\rho_i+\frac{\beta}{\rho^2}\rho_i\rho_j-\frac{\beta}{\rho}\rho_{ij}.
\end{aligned}
\end{equation}
Plug \eqref{2.100}, \eqref{2.104} and \eqref{2.106} into \eqref{2.107}, we get
\begin{equation}\label{2.108}
\beta_{ij}=\sum_kh_{ijk}\rho_k-\frac{1}{\rho}\rho_i\rho_kh_{jk}-\frac{1}{\rho}\rho_j\rho_kh_{ik}+\frac{3\beta}{\rho^2}\rho_i\rho_j+\frac{1+\beta^2}{\rho}h_{ij}-\beta\sum_kh_{ik}h_{kj}-\frac{\beta}{\rho^2}\delta_{ij}.
\end{equation}

Now we are ready to derive a bound for principal curvatures. Set
  \begin{equation}\label{2.55}
  M:=\max_{\Sigma}\frac{\kappa_{\max}(X)}{1-e^{-A\beta(X)}},
  \end{equation}
  where $\kappa_{\max}(X)$ means the largest principal curvature of $\Sigma$ at $X$ and $A$ is a positive constant to be chosen later. It suffices to derive a bound for $M$. If $M$ is achieved on $\partial\Sigma$, by the established $C^2$ bound on the boundary we are done.

  Therefore we just assume $M$ is achieved at an interior point $X_0\in\Sigma$. We choose the local orthonormal frame $\{\tau_1,\tau_2,...,\tau_n\}$ around $X_0$ such that $h_{ij}$ is diagonal at $X_0$, i.e., $h_{ij}(X_0)=\kappa_i\delta_{ij}$, and $h_{11}(X_0)=\kappa_1$ is the largest principal curvature at $X_0$. Then we shall note at $X_0$ the formulas \eqref{2.101}, \eqref{2.106} and \eqref{2.108} can be simplified as follows,
\begin{equation}\label{2.109}
h_{iijj}-h_{jjii}=(\kappa_i-\kappa_j)\kappa_i\kappa_j,
\end{equation}
\begin{equation}\label{2.110}
\beta_i=\kappa_i\rho_i-\frac{\beta}{\rho}\rho_i,
\end{equation}
\begin{equation}\label{2.111}
\beta_{ii}=\sum_kh_{iik}\rho_k+\frac{3\beta}{\rho^2}\rho_i^2-\frac{\beta}{\rho^2}+\frac{1+\beta^2}{\rho}\kappa_i-\frac{2}{\rho}\rho_i^2\kappa_i-\beta\kappa_i^2.
\end{equation}

In the rest of this subsection all the computations are calculated at $X_0$. By our assumption the function $\ln\left(\frac{h_{11}}{1-e^{-A\beta}}\right)$ achieves its local maximum at $X_0$. Therefore we have
\begin{equation}\label{2.112}
0=\bar\nabla_i\ln\left(\frac{h_{11}}{1-e^{-A\beta}}\right)
=\frac{h_{11i}}{h_{11}}-\frac{A\beta_i}{e^{A\beta}-1},
\end{equation} 
\begin{equation}\label{2.113}
0\geq\bar\nabla_{ii}\ln\left(\frac{h_{11}}{1-e^{-A\beta}}\right)
=\frac{h_{11ii}}{h_{11}}-\frac{h_{11i}^2}{h_{11}^2}-\frac{A\beta_{ii}}{e^{A\beta}-1}+\frac{A^2e^{A\beta}\beta_i^2}{(e^{A\beta}-1)^2}.
\end{equation}
We plug \eqref{2.109}-\eqref{2.112} into \eqref{2.113}. Then we can get
\begin{equation}\label{2.114}
\begin{aligned}
h_{ii11}\leq& \kappa_i^2\kappa_1-\kappa_1^2\kappa_i+\frac{A\kappa_1}{e^{A\beta}-1}\left(\sum_kh_{iik}\rho_k+\frac{3\beta}{\rho^2}\rho_i^2-\frac{\beta}{\rho^2}+\frac{1+\beta^2}{\rho}\kappa_i-\frac{2}{\rho}\rho_i^2\kappa_i-\beta\kappa_i^2\right)\\
&-\frac{A^2\kappa_1}{e^{A\beta}-1}\left(\kappa_i^2\rho_i^2+\frac{\beta^2}{\rho^2}\rho_i^2-\frac{2\beta}{\rho}\kappa_i\rho_i^2\right).
\end{aligned}
\end{equation}

Next we shall differentiate the equation
\begin{equation}\label{2.115}
F([h_{ij}])=\psi,
\end{equation}
where $F$ is defined as in \eqref{1.6}. We get
\begin{equation}\label{2.116}
\sum_{i,j}F^{ij}h_{ijk}=\psi_k.
\end{equation}
Choose $k=1$ in \eqref{2.116} and differentiate it again by $\tau_1$, we get
\begin{equation}\label{2.118}
\sum_{i,j,k,l}F^{ij,kl}h_{kl1}h_{ij1}+\sum_{i,j}F^{ij}h_{ij11}=\psi_{11}.
\end{equation}
Since at $X_0$ $h_{ij}$ is diagonal, $F^{ij}$ is also diagonal and $F^{ij}=f_i\delta_{ij}$. We also note that $F$ is concave. Therefore \eqref{2.116} and \eqref{2.118} can be simplified as
\begin{equation}\label{2.117}
\sum_if_ih_{iik}=\psi_k,
\end{equation}
\begin{equation}\label{2.119}
\psi_{11}\leq\sum_{i}f_ih_{ii11}.
\end{equation}

Combining \eqref{2.114} and \eqref{2.119}, we get
\begin{equation}\label{2.120}
\begin{aligned}
\psi_{11}\leq&\kappa_1\sum_if_i\kappa_i^2-\kappa_1^2\sum_if_i\kappa_i+\frac{A\kappa_1}{e^{A\beta}-1}\sum_{i,k}f_ih_{iik}\rho_k\\
&+\frac{A\kappa_1}{e^{A\beta}-1}\sum_i\left(\frac{3\beta}{\rho^2}f_i\rho_i^2-\frac{\beta}{\rho^2}f_i+\frac{1+\beta^2}{\rho}f_i\kappa_i-\frac{2}{\rho}f_i\kappa_i\rho_i^2-\beta f_i\kappa_i^2\right)\\
&-\frac{A^2\kappa_1}{e^{A\beta}-1}\sum_i\left(f_i\kappa_i^2\rho_i^2+\frac{\beta^2}{\rho^2}f_i\rho_i^2-\frac{2\beta}{\rho}f_i\kappa_i\rho_i^2\right).
\end{aligned}
\end{equation}
From \eqref{2.117} we see the term $\displaystyle\sum_{i,k}f_ih_{iik}\rho_k$ in \eqref{2.120} can replaced by $\displaystyle\sum_k\psi_k\rho_k$. Rearranging terms in \eqref{2.120} we get
\begin{equation}\label{2.121}
\begin{aligned}
\psi_{11}\leq& \left(\kappa_1-\frac{A\beta\kappa_1}{e^{A\beta}-1}\right)\sum_if_i\kappa_i^2-\kappa_1^2\sum_if_i\kappa_i+\frac{(1+\beta^2)A\kappa_1}{\rho\left(e^{A\beta}-1\right)}\sum_if_i\kappa_i\\
&+\frac{A\kappa_1}{e^{A\beta}-1}\sum_k\psi_k\rho_k-\frac{A\beta\kappa_1}{\rho^2\left(e^{A\beta}-1\right)}\sum_if_i+\frac{A\beta\kappa_1(3-A\beta)}{\rho^2\left(e^{A\beta}-1\right)}\sum_if_i\rho_i^2\\
&+\frac{2A\kappa_1(A\beta-1)}{\rho\left(e^{A\beta}-1\right)}\sum_if_i\kappa_i\rho_i^2-\frac{A^2\kappa_1}{e^{A\beta}-1}\sum_if_i\kappa_i^2\rho_i^2.
\end{aligned}
\end{equation}
Since $\beta$ has a positive lower bound, we can choose $A$ large enough to ensure $3-A\beta<0$. Then in \eqref{2.121} we throw away some negative terms in the right hand side of the inequality, which are the 5th, 6th and 8th terms, getting
\begin{equation}\label{2.122}
\begin{aligned}
&\kappa_1^2\sum_if_i\kappa_i+\left(\frac{A\beta}{e^{A\beta}-1}-1\right)\kappa_1\sum_if_i\kappa_i^2\\
\leq&-\psi_{11}+\left(\frac{A(1+\beta^2)}{\rho(e^{A\beta}-1)}\sum_if_i\kappa_i+\frac{A}{e^{A\beta}-1}\sum_k\psi_k\rho_k+\frac{2A(A\beta-1)}{\rho(e^{A\beta}-1)}\sum_if_i\kappa_i\rho_i^2\right)\kappa_1.
\end{aligned}
\end{equation}
Since $\kappa_1$ is the largest principal curvature at $X_0$, we have $\displaystyle\kappa_1\sum_if_i\kappa_i^2\leq\kappa_1^2\sum_if_i\kappa_i$. We shall also note that $\frac{A\beta}{e^{A\beta}-1}-1\leq0$. Therefore
\begin{equation}\label{2.123}
\left(\frac{A\beta}{e^{A\beta}-1}-1\right)\kappa_1\sum_if_i\kappa_i^2\geq\left(\frac{A\beta}{e^{A\beta}-1}-1\right)\kappa_1^2\sum_if_i\kappa_i.
\end{equation}
Combining \eqref{2.122} and \eqref{2.123}, we get
\begin{equation}\label{2.124}
\begin{aligned}
&\left(\frac{A\beta}{e^{A\beta}-1}\sum_if_i\kappa_i\right)\kappa_1^2\\
\leq&-\psi_{11}+\left(\frac{A(1+\beta^2)}{\rho(e^{A\beta}-1)}\sum_if_i\kappa_i+\frac{A}{e^{A\beta}-1}\sum_k\psi_k\rho_k+\frac{2A(A\beta-1)}{\rho(e^{A\beta}-1)}\sum_if_i\kappa_i\rho_i^2\right)\kappa_1.
\end{aligned}
\end{equation}
From condition \eqref{0.5}, \eqref{0.6} and \eqref{0.7},
  \begin{equation}\label{2.86}
  \sigma_0\leq\sum_if_i\kappa_i\leq f(\kappa_1,...,\kappa_n)=\psi\leq C,
  \end{equation}
  where $\sigma_0$ and $C$ are uniform positive constants. It is also straight forward to see that
  \begin{equation}\label{2.87}
  \tilde\nabla_{11}\psi-\bar\nabla_{11}\psi=-(\tilde\nabla_{\tau_1}\tau_1-\bar\nabla_{\tau_1}\tau_1)\psi=-h_{11}\nu(\psi)=-\kappa_1\nu(\psi).
  \end{equation}
  Therefore from the smoothness of $\psi$ we see that $|\psi_{11}|$ is controlled by $C(1+\kappa_1)$. As for the second long term in the right hand side of \eqref{2.124}, since $\rho$, $\beta$ are all well controlled terms by previous estimates, and it is easily seen that $|\psi_i|<C$, $|\rho_i|\leq |\tilde\nabla\rho|=1$, combining with \eqref{2.86} we see that it is controlled by $C\kappa_1$. Recall the definition \eqref{2.55} of $M$, then \eqref{2.124} implies
  \begin{equation}\label{2.125}
  c_0M^2\leq C(1+M)
  \end{equation}
  for some uniform positive constants $c_0$ and $C$, which yields an upper bound for $M$. As we said at the beginning of this subsection, this gives an upper bound for principal curvatures, and hence $|\nabla^2u|$ in $\Omega$. The global $C^2$ bound is established. We also note by compactness argument the upper bound for principal curvatures implies a positive lower bound for principal curvatures, which will also be used when we prove the existence. We combine all these estimates in the following theorem.
  \begin{thm}\label{2.91}
  Let $u\geq\underline{u}$ be a strictly locally convex solution of \eqref{1.14} and $\Sigma$: $X=\frac{1}{u}x$ the corresponding radial graph. $\kappa_i$ is the principal curvature of $\Sigma$. Then we have the following estimates:
  \begin{equation}\label{2.92}
  \|u\|_{C^2(\bar\Omega)}\leq C,\quad C^{-1}\leq\kappa_i\leq C,
  \end{equation}
  where $C$ is a positive constant depending on $\Omega$, $\inf_\Omega\underline{u}$, $\|\underline{u}\|_{C^4(\bar\Omega)}$, $\|\psi\|_{C^2(\Lambda_K)}$, $\inf_{\Lambda_K}\psi$ and the convexity of $\underline{u}$.
  \end{thm}
  \hspace{3cm}
  \section{Existence}\label{3.1}
  In this section we apply the classical method of continuity (see \cite{trudinger}) and the degree theory in \cite{degreethm} developed by Y.Y. Li to establish the existence of solution of \eqref{0.2}-\eqref{0.3}. As noted at the end of Section \ref{1.18}, we shall work on equation \eqref{1.21}. Precisely, we work on two auxiliary forms of \eqref{1.21}, that is,
  \begin{equation}\label{3.5}
  \begin{aligned}
  H(\nabla^2v,\nabla v,v)&=\left(t\epsilon+(1-t)\frac{\underline{\psi}(x)}{e^{2\underline{v}}}\right)e^{2v}\quad\text{in}\:\Omega,\\
  v&=\underline{v}\quad\text{on}\:\partial\Omega
  \end{aligned}
  \end{equation}
  and
  \begin{equation}\label{3.7}
  \begin{aligned}
  H(\nabla^2v,\nabla v,v)&=t\psi(X(x))+(1-t)\epsilon e^{2v}\quad\text{in}\:\Omega,\\
  v&=\underline{v}\quad\text{on}\:\partial\Omega,
  \end{aligned}
  \end{equation}
  where $t\in[0,1]$ and $\epsilon$ is a fixed small number such that
  \begin{equation}\label{3.8}
  \underline{\psi}(x)>\psi(\bar X(x))+\epsilon K^2\quad\text{in}\:\bar\Omega.
  \end{equation}

  Before going to the proof of existence, we need some preparation. We first introduce an important property of the operator $H$ in \eqref{1.21}, which, compared with \eqref{2.9} and \eqref{2.19}, explains why we use $H$ instead of $G$.

  \begin{lem}\label{3.2}
  Let $v$ be a strictly locally convex solution of $H(\nabla^2v,\nabla v,v)=\psi$, then $H_v:=\frac{\partial H}{\partial v}\leq\psi$.
  \end{lem}

  \begin{proof}
  From \eqref{1.19}-\eqref{1.21}, it is easily seen that
  \begin{equation}\label{3.3}
  H_v=F^{ij}a_{ij}.
  \end{equation}
  Recall the properties of $F$ introduced in Section \ref{1.18} and the concavity of $f$, we get
  \begin{equation}\label{3.4}
  H_v=f_i\kappa_i\leq f(\kappa_1,\kappa_2,...,\kappa_n)=\psi.
  \end{equation}
  \end{proof}

  \begin{lem}\label{3.9}
  For any $t\in[0,1]$, \eqref{3.5} has at most one strictly locally convex solution v, and $v\geq\underline{v}$.
  \end{lem}

  \begin{proof}
  We just give the proof that $v\geq\underline{v}$. The uniqueness follows almost the same argument. Suppose not, then $\underline{v}-v$ achieves positive maximum in some interior point $x_0\in\Omega$. We have
  \begin{equation} \label{3.16}
  \underline{v}(x_0)>v(x_0),\quad\nabla\underline{v}(x_0)=\nabla v(x_0),\quad \nabla^2\underline{v}(x_0)\leq\nabla^2v(x_0).
  \end{equation}
  Consider the deformation $s\underline{v}+(1-s)v$ near $x_0$,
  \begin{equation}\label{3.10}
  \begin{aligned}
  &\delta_{ij}+\nabla_i\left(s\underline{v}+(1-s)v\right)\nabla_j\left(s\underline{v}+(1-s)v\right)+\nabla_{ij}\left(s\underline{v}+(1-s)v\right)|_{x_0}\\
  =&\delta_{ij}+\nabla_i\underline{v}\nabla_j\underline{v}+\nabla_{ij}\underline{v}+(1-s)\nabla_{ij}(v-\underline{v})|_{x_0}\\
  >&0\quad\forall s\in[0,1].
  \end{aligned}
  \end{equation}
  So we can define a differentiable function on $[0,1]$,
  \begin{equation}\label{3.11}
  \begin{aligned}
  a(s):=&H\left(\nabla^2\left(s\underline{v}+(1-s)v\right),\nabla \left(s\underline{v}+(1-s)v\right),s\underline{v}+(1-s)v\right)(x_0)\\
  &-\left(t\epsilon+(1-t)\frac{\underline{\psi}(x_0)}{e^{2\underline{v}(x_0)}}\right)e^{2\left(s\underline{v}(x_0)+(1-s)v(x_0)\right)}.
  \end{aligned}
  \end{equation}
  Note
  \begin{equation}\label{3.12}
  \begin{aligned}
  a(0)&=H(\nabla^2v,\nabla v,v)(x_0)-\left(t\epsilon+(1-t)\frac{\underline{\psi}(x_0)}{e^{2\underline{v}(x_0)}}\right)e^{2v(x_0)}\\
  &=0,\\
  a(1)&=H(\nabla^2\underline{v},\nabla\underline{v},\underline{v})(x_0)-\left(t\epsilon+(1-t)\frac{\underline{\psi}(x_0)}{e^{2\underline{v}(x_0)}}\right)e^{2\underline{v}(x_0)}\\
  &=\underline{\psi}(x_0)-\left(\epsilon te^{2\underline{v}(x_0)}+(1-t)\underline{\psi}(x_0)\right)\\
  &=t\left(\underline{\psi}(x_0)-\epsilon e^{2\underline{v}(x_0)}\right)\\
  &\geq0.
  \end{aligned}
  \end{equation}
  Then there exists $s_0\in[0,1]$ such that $a(s_0)=0$, $a'(s_0)\geq0$, that is,
  \begin{equation}\label{3.13}
  \begin{aligned}
  &H\left(\nabla^2\left(s_0\underline{v}+(1-s_0)v\right),\nabla \left(s_0\underline{v}+(1-s_0)v\right),s_0\underline{v}+(1-s_0)v\right)(x_0)\\
  =&\left(t\epsilon+(1-t)\frac{\underline{\psi}(x_0)}{e^{2\underline{v}(x_0)}}\right)e^{2(s_0\underline{v}(x_0)+(1-s_0)v(x_0))},
  \end{aligned}
  \end{equation}
  \begin{equation}\label{3.14}
  \begin{aligned}
  &H^{ij}|_{s_0\underline{v}(x_0)+(1-s_0)v(x_0)}\nabla_{ij}(\underline{v}-v)(x_0)+H^i|_{s_0\underline{v}(x_0)+(1-s_0)v(x_0)}\nabla_i(\underline{v}-v)(x_0)\\
  &+\left(H_v|_{s_0\underline{v}(x_0)+(1-s_0)v(x_0)}-2\left(t\epsilon+(1-t)\frac{\underline{\psi}(x_0)}{e^{2\underline{v}(x_0)}}\right)e^{2\left(s_0\underline{v}(x_0)+(1-s_0)v(x_0)\right)}\right)(\underline{v}-v)(x_0)\geq0.
  \end{aligned}
  \end{equation}
  But Lemma \ref{3.2} and \eqref{3.13} imply
  \begin{equation}\label{3.15}
  H_v|_{s_0\underline{v}(x_0)+(1-s_0)v(x_0)}\leq\left(t\epsilon+(1-t)\frac{\underline{\psi}(x_0)}{e^{2\underline{v}(x_0)}}\right)e^{2(s_0\underline{v}(x_0)+(1-s_0)v(x_0))}.
  \end{equation}
  Combining with \eqref{3.16} and the ellipticity of $H$, we can see in the left hand side of \eqref{3.14} the first term is nonpositive, the second term 0 and the last term negative. Thus the left hand side is strictly less than 0, which is a contradiction.
  \end{proof}

  \begin{lem}\label{3.17}
  Let $v\geq\underline{v}$ be a strictly locally convex solution of \eqref{3.7}, then $v>\underline{v}$ in $\Omega$, $\mathbf{n}(v-\underline{v})>0$ on $\partial\Omega$, where $\mathbf{n}$ is the interior unit normal of $\partial\Omega$.
  \end{lem}

  \begin{proof}
  We show both parts by contradiction. Suppose $v=\underline{v}$ at some point $x_0\in\Omega$, then $x_0$ is a local minimum of $v-\underline{v}$. So we have
  \begin{equation}\label{3.18}
  v(x_0)=\underline{v}(x_0),\quad\nabla v(x_0)=\nabla\underline{v}(x_0),\quad\nabla^2v(x_0)\geq\nabla^2\underline{v}(x_0).
  \end{equation}
  Then by the formula \eqref{1.19} for $[a_{ij}]$, clearly
  \begin{equation}\label{3.19}
  a_{ij}[v](x_0)\geq a_{ij}[\underline{v}](x_0).
  \end{equation}
  So
  \begin{equation}\label{3.20}
  H(\nabla^2v,\nabla v,v)(x_0)=F\left(a_{ij}[v](x_0)\right)\geq F\left(a_{ij}[\underline{v}](x_0)\right)=H(\nabla^2\underline{v},\nabla \underline{v},\underline{v})(x_0).
  \end{equation}
  However, by the choice of $\epsilon$ in \eqref{3.8}, we can see
  \begin{equation}\label{3.21}
  \begin{aligned}
  &H(\nabla^2v,\nabla v,v)(x_0)\\
  =&t\psi(X(x_0))+(1-t)\epsilon e^{2v(x_0)}\\
  =&t\psi(\bar X(x_0))+(1-t)\epsilon e^{2v(x_0)}\\
  <&\underline{\psi}(x_0)\\
  =&H(\nabla^2\underline{v},\nabla \underline{v},\underline{v})(x_0),
  \end{aligned}
  \end{equation}
  which is a contradiction.

  From the fact $v\geq\underline{v}$ we can see $\mathbf{n}(v-\underline{v})\geq0$ on $\partial\Omega$. Suppose $\mathbf{n}(v-\underline{v})=0$ at some point $x_0\in\partial\Omega$. Then by $v=\underline{v}$ on $\partial\Omega$ we get $\nabla v(x_0)=\nabla\underline{v}(x_0)$. From \eqref{3.21} we can see $a_{ij}[v](x_0)\geq a_{ij}[\underline{v}](x_0)$ can not hold, and thus $\nabla^2v(x_0)\geq\nabla^2\underline{v}(x_0)$ can not hold. There exists unit vector $\xi\in T_{x_0}\mathbb{S}^n$ such that $\nabla_{\xi\xi}v(x_0)<\nabla_{\xi\xi}\underline{v}(x_0)$. Again by the fact that $v=\underline{v}$ on $\partial\Omega$ and $\nabla v(x_0)=\nabla\underline{v}(x_0)$, $\xi$ can not be tangential to $\partial\Omega$ at $x_0$. We can assume $\xi$ (or $-\xi$) pointing to the interior of $\Omega$. Let $c(t)$ be the normalized geodesic starting at $x_0$ in the direction of $\xi$. In a short time $c$ stays inside $\Omega$. We compare $v(c(t))$ and $\underline{v}(c(t))$,
  \begin{equation}\label{3.22}
  \begin{aligned}
  v\circ c(0)&=\underline{v}\circ c(0),\\
  (v\circ c)'(0)=\nabla_\xi v(x_0)&=\nabla_\xi\underline{v}(x_0)=(\underline{v}\circ c)'(0),\\
  (v\circ c)''(0)=\nabla_{\xi\xi} v(x_0)&<\nabla_{\xi\xi}\underline{v}(x_0)=(\underline{v}\circ c)''(0).
  \end{aligned}
  \end{equation}
  Therefore in a short time $v(c(t))<\underline{v}(c(t))$. But it is contradicted with $v>\underline{v}$ in $\Omega$.
  \end{proof}

  Now we are ready to prove the existence of solution for \eqref{3.5}, and then \eqref{3.7}.

  \begin{thm}\label{3.23}
  For any $t\in[0,1]$, \eqref{3.5} has a unique strictly locally convex solution.
  \end{thm}

  \begin{proof}
  Uniqueness is already proved in Lemma \ref{3.9}. We just prove the existence, with the standard continuity method.

  We shall first establish a priori estimates for \eqref{3.5}. Since \eqref{1.14} and \eqref{3.5} are related by changing variable $u=e^v$ and note that we also have positive lower bound for $u$, Theorem \ref{2.91} directly gives $C^2$ estimates for strictly locally convex solutions $v$ with $v\geq\underline{v}$. The uniform positive upper and lower bounds for principal curvatures ensure that \eqref{3.5} is uniformly elliptic for strictly locally convex solutions $v$ with $v\geq\underline{v}$. We also pointed out $H$ is a concave operator. Then by the Evans-Krylov estimates \cite{evans} and \cite{krylov}, we get $C^{2,\alpha}$ estimates for some $\alpha\in(0,1)$, that is
  \begin{equation}\label{3.24}
  \|v\|_{C^{2,\alpha}(\bar\Omega)}\leq C.
  \end{equation}
  Here we shall note $C$ is independent of $t$.

  Let $C_0^{2,\alpha}(\bar\Omega)$ be the subspace of $C^{2,\alpha}(\bar\Omega)$ consisting of functions vanishing on the boundary. Consider $\mathcal{U}=\{w\in C_0^{2,\alpha}(\bar\Omega)|w+\underline{v}\text{ is strictly locally convex}\}$, which is open in $C_0^{2,\alpha}(\bar\Omega)$. Construct a map $L$ from $\mathcal{U}\times[0,1]$ to $C^\alpha(\bar\Omega)$ by
  \begin{equation}\label{3.25}
  L[w,t]=H\left(\nabla^2(w+\underline{v}),\nabla(w+\underline{v}),w+\underline{v}\right)-\left(t\epsilon+(1-t)\frac{\underline{\psi}(x)}{e^{2\underline{v}}}\right)e^{2(w+\underline{v})}.
  \end{equation}
  Set $S=\{t\in[0,1]|L[w,t]=0\:\text{has a solution in }\mathcal{U}\}$.

  First, $L[0,0]=0$ since $\underline{v}$ is clearly a solution of \eqref{3.5} when $t=0$. So $0\in S$ and $S$ is not empty.

  Second, for any $t_0\in S$, there exists $w_0\in\mathcal{U}$ such that $L[w_0,t_0]=0$. The Fr\'echet derivative of $L$ with respect to $w$ at $(w_0,t_0)$ is a linear elliptic operator from $C_0^{2,\alpha}(\bar\Omega)$ to $C^\alpha(\bar\Omega)$,
  \begin{equation}\label{3.26}
  \begin{aligned}
  &L_w|_{(w_0,t_0)}(h)\\
  =&H^{ij}|_{w_0+\underline{v}}\nabla_{ij}h+H^i|_{w_0+\underline{v}}\nabla_ih
  +\left(H_v|_{w_0+\underline{v}}-2\left(t_0\epsilon+(1-t_0)\frac{\underline{\psi}(x)}{e^{2\underline{v}}}\right)e^{2(w_0+\underline{v})}\right)h.
  \end{aligned}
  \end{equation}
  By Lemma \ref{3.2}, we can see $H_v|_{w_0+\underline{v}}-2\left(t_0\epsilon+(1-t_0)\frac{\underline{\psi}(x)}{e^{2\underline{v}}}\right)e^{2(w_0+\underline{v})}<0$. Therefore by standard elliptic theory $L_w|_{(w_0,t_0)}$ is invertible. By implicit function theory, a neighbourhood of $t_0$ is contained in $S$. $S$ is open in $[0,1]$.

  Third, let ${t_i}$ be a sequence in $S$ converging to $t_0\in[0,1]$ and $w_i$ the corresponding solution with respect to $t_i$. By Lemma \ref{3.9}, $w_i\geq0$. Then we can apply estimates \eqref{3.24} to see $v_i=w_i+\underline{v}$ is a bounded sequence in $C^{2,\alpha}(\bar\Omega)$. Sending $t_i$ to $t_0$, passing to a subsequence if necessary,  we get a limit function $v_0$ which is the solution of \eqref{3.5} at $t_0$. From the uniform upper and lower bounds for principal curvature which are independent of $t$, we can see $v_0$ is strictly locally convex. Above all, $w_0=(v_0-\underline{v})\in\mathcal{U}$ and $L[w_0,t_0]=0$. So $t_0\in S$ and $S$ is closed in $[0,1]$.

  We proved $S$ is a nonempty and both open and closed subset of $[0,1]$. Therefore $S=[0,1]$. \eqref{3.5} has a strictly locally convex solution for any $t\in[0,1]$.
  \end{proof}

  \begin{thm}\label{3.27}
  For any $t\in[0,1]$, \eqref{3.7} has a strictly locally convex solution. In particular, \eqref{0.2}-\eqref{0.3} has a strictly locally convex solution.
  \end{thm}

  \begin{proof}
  Similarly as in the proof of Theorem \ref{3.23}, we get $C^{2,\alpha}$ estimates for strictly locally convex solutions of \eqref{3.7} with $v\geq\underline{v}$. Then by the standard regularity theory for second order uniformly elliptic equations, we can get any higher order estimates. Here we need $C^{4,\alpha}$ estimates for applying the degree theory in \cite{degreethm}, that is,
  \begin{equation}\label{3.28}
  \|v\|_{C^{4,\alpha}(\bar\Omega)}<C_1.
  \end{equation}
   We also need to describe the uniform bounds for principal curvatures more specifically, that is,
  \begin{equation}\label{3.29}
  C_2^{-1}I<[\delta_{ij}+\nabla_iv\nabla_jv+\nabla_{ij}v]<C_2I\text{ in }\bar\Omega.
  \end{equation}
   We shall note that both $C_1$ and $C_2$ are uniformly positive constants which are independent of $t$.

  Let $C_0^{4,\alpha}(\bar\Omega)$ be the subspace of $C^{4,\alpha}(\bar\Omega)$ consisting of functions vanishing on the boundary. Consider $\mathcal{O}=\{w\in C_0^{4,\alpha}(\bar\Omega)|w>0\text{ in }\Omega,\:\mathbf{n}w>0\text{ on }\partial\Omega,\:C_2^{-1}I<[\delta_{ij}+\nabla_i(w+\underline{v})\nabla_j(w+\underline{v})+\nabla_{ij}(w+\underline{v})]<C_2I\text{ in }\bar\Omega,\:\|w\|_{C^{4,\alpha}(\bar\Omega)}<C_1+\|\underline{v}\|_{C^{4,\alpha}(\bar\Omega)}\}$, where $C_1$ and $C_2$ are as in \eqref{3.28} and \eqref{3.29} and $\mathbf{n}$ is the unit interior normal on $\partial\Omega$. $\mathcal{O}$ is a bounded open subset of $C_0^{4,\alpha}(\bar\Omega)$. Construct a map from $\mathcal{O}\times[0,1]$ to $C^{2,\alpha}(\bar\Omega)$:
  \begin{equation}\label{3.30}
  M_t[w]=H\left(\nabla^2(w+\underline{v}),\nabla(w+\underline{v}),w+\underline{v}\right)-t\psi(e^{-(w+\underline{v})}x)-(1-t)\epsilon e^{2(w+\underline{v})}
  \end{equation}
   From Theorem \ref{3.23}, let $v^0$ be the unique solution of \eqref{3.5} at $t=1$. Set $w^0=v^0-\underline{v}$. By Lemma \ref{3.9}, $w^0\geq0$. Then by Lemma \ref{3.17}, $w^0>0$ in $\Omega$ and $\mathbf{n}w^0>0$ on $\partial\Omega$. Also clearly $v^0$ satisfies \eqref{3.28} and \eqref{3.29}. Therefore $w^0\in\mathcal{O}$.

    It is easy to check that $M_t[w]=0$ has no solution on $\partial\mathcal{O}$. Namely, if $w$ with $C^{4,\alpha}$ norm $C_1+\|\underline{v}\|_{C^{4,\alpha}(\bar\Omega)}$ is a solution, it contradicts with estimate \eqref{3.28}. If $w=0$ at some interior point or $\mathbf{n}w=0$ at some boundary point, it will contradict with Lemma \ref{3.17}. If $[\delta_{ij}+\nabla_i(w+\underline{v})\nabla_j(w+\underline{v})+\nabla_{ij}(w+\underline{v})]$ achieves eigenvalue $C_2$ or $C_2^{-1}$ at some point, it would contradict with the estimates \eqref{3.29}. Above all, $M_t[w]=0$ has no solution on $\partial\mathcal{O}$ for any $t$. We also note $M_t$ is uniformly elliptic on $\mathcal{O}$, independent of t. Therefore, the degree of $M_t$ on $\mathcal{O}$ at $0$ $\text{deg}(M_t,\mathcal{O},0)$ is well defined and independent of $t$.

   Now we compute $\text{deg}(M_0,\mathcal{O},0)$. $M_0[w]=0$ has a unique solution $w^0$ in $\mathcal{O}$. The Fr\'echet derivative of $M_0$ at $w^0$ is a linear elliptic operator from $C^{4,\alpha}_0(\bar\Omega)$ to $C^{2,\alpha}(\bar\Omega)$,
   \begin{equation}\label{3.31}
   M_{0,w^0}(h)=H^{ij}|_{v^0}\nabla_{ij}h+H^i|_{v^0}\nabla_ih+\left(H_v|_{v^0}-2\epsilon e^{2v^0}\right)h.
   \end{equation}
   By Lemma \ref{3.2}, $H_v|_{v^0}-2\epsilon e^{2v^0}<0$. So $M_{0,w^0}$ is invertible. By the theory in \cite{degreethm}, we can see
   \begin{equation}\label{3.32}
   \text{deg}(M_0,\mathcal{O},0)=\text{deg}(M_{0,w^0},B_1,0)=\pm1\neq0,
   \end{equation}
   where $B_1$ is the unit ball of $C^{4,\alpha}_0(\bar\Omega)$. Therefore
   \begin{equation}\label{3.33}
   \text{deg}(M_t,\mathcal{O},0)\neq0\text{ for all }t\in[0,1].
   \end{equation}
   \eqref{3.7} has at least one strictly locally convex solution for any $t\in[0,1]$. In particular, when $t=1$, it solves \eqref{0.2}-\eqref{0.3}.
  \end{proof}

  \hspace{3cm}
  
\noindent 
Department of Mathematics, Johns Hopkins University, 3400 N. Charles ST, Baltimore, MD 21218\\
Email: csu8@jhu.edu
  

\newpage

  \begin{thebibliography}{9}
  
\bibitem{CNS1}
Caffarelli, L.; Nirenberg, L.; Spruck, J.,
\emph{The Dirichlet problem for nonlinear second-order elliptic equations I. Monge-Amp{\`e}re equation},
Comm. Pure Appl. Math. 37 (1984), no. 3, 369-402.

\bibitem{CNS3}
Caffarelli, L.; Nirenberg, L.; Spruck, J.,
\emph{The Dirichlet problem for nonlinear second-order elliptic equations. III. Functions of the eigenvalues of the Hessian},
Acta Math. 155 (1985), no. 3-4, 261-301. 
  
  
  \bibitem{clarkesmith}
Clarke, A.; Smith, G.,
\emph{The Perron method and the non-linear Plateau problem},
Geom. Dedicata. 163(2013), 159-164.
  
  \bibitem{evans}
  Evans, L.C.,
  \emph{Classical solutions of fully nonlinear, convex, second order elliptic equations},
  Comm. Pure Appl. Math. 35 (1982), 333-363. 
  
  \bibitem{ghomiphdthesis}
Ghomi, M.,
\emph{Strictly convex submanifolds and hypersurfaces of positive curvature},
J. Differential Geom. 57(2001), no. 2, 239-271.
  
  \bibitem{trudinger}
  Gilbarg, D.; Trudinger, N.S.,
  \emph{Elliptic partial differential equations of second order}, 2nd edition,
  Springer-Verlag, New York, 1983.
  
  \bibitem{Guannonconvex}
Guan, B.,
\emph{The Dirichlet problem for Monge-Amp\`ere equations in non-convex domains and spacelike hypersurfaces of constant Gauss curvature},
Trans. Amer. Math. Soc. 350(1998), no. 12, 4955-4971.
  
  \bibitem{radialgs}
  Guan, B.; Spruck, J.,
  \emph{Boundary-value problem on $S^n$ for surfaces of constant Gauss curvature},
  Ann. of Math. (2) 138 (1993), no. 3, 601-624.

\bibitem{gsgaussperron} 
Guan, B.; Spruck, J.,
\emph{The existence of hypersurfaces of constant Gauss curvature with prescribed boundary},
J. Differential Geom. 62(2002), no. 2, 259-287.

  \bibitem{locallyconvexgs}
  Guan, B.; Spruck, J.,
  \emph{Locally convex hypersurfaces of constant curvature with boundary},
  Comm. Pure Appl. Math. 57 (2004), 1311-1331.
  

  
  \bibitem{heijenoort}
van Heijenoort, J.,
\emph{On locally convex manifolds},
Comm. Pure Appl. Math. 5(1952), 223-242.
  
  \bibitem{hoffmanrosenbergspruck}
Hoffman, D.; Rosenberg, H.; Spruck, J.,
\emph{Boundary value problems for surfaces of constant Gauss curvature},
Comm. Pure Appl. Math. 45(1992), no. 8, 1051-1062.

  
  \bibitem{krylov}
  Krylov, N.V.,
  \emph{Boundedly nonhomogeneous elliptic and parabolic equations in a domain},
  Izvestia Math. Ser. 47 (1983), 75-108.
  
  \bibitem{degreethm}
  Li, Y.Y.,
  \emph{Degree theory for second order nonlinear elliptic operators and its applications},
  Comm. in PDE's. 14(11) (1989), 1541-1579. 
  
  \bibitem{rosenberg}
  Rosenberg, H.,
  \emph{Hypersurfaces of constant curvature in space forms},
  Bull. Sci. Math. 117(1993), no. 2, 211-239.

\bibitem{schoensimonyau}
Schoen, R.; Simon, L.; Yau, S.T.,
\emph{Curvature estimates for minimal hypersurfaces},
Acta Math. 134 (1975), 275-288.

\bibitem{smith1}
Smith, G.,
\emph{The Plateau problem for general curvature functions},
arXiv:1008.3545

\bibitem{smith2}
Smith, G.,
\emph{Compactness results for immersions of prescribed Gaussian curvature II - geometric aspects},
Geom. Dedicata 172, no. 1, (2014), 303-350.

\bibitem{trudingerwang}
Trudinger, N.S.; Wang, X.-J.,
\emph{On locally convex hypersurfaces with boundary}, 
J. Reine Angew. Math. 551(2002), 11-32.

  \end{thebibliography}
\end{document}